\documentclass[11pt]{article}

\usepackage{amssymb}
\usepackage{graphicx}
\usepackage{color}
\usepackage{amsmath,amsthm,amsfonts,epsfig,setspace}
\numberwithin{equation}{section}

\newcommand{\ds}{\displaystyle}
\def\nm{\noalign{\medskip}}
\newtheorem{thm}{Theorem}[section]
\newtheorem{rmk}{Remark}[section]
\newtheorem{cor}{Corollary}[section]
\newtheorem{definition}{Definition} [section]

\newtheorem{prop}{Proposition}[section]

\newtheorem{cond}{Condition}

 \def\p{\partial}
\def \Vh0{\stackrel{\circ}{V}_h}

\def\l{\label}  \def\f{\frac} \def\df{\dfrac} 

   \def\eps{\varepsilon}

\def\l|{\left|}
\def\r|{\right|}

\newcommand{\R}{\mathbb{R}}

\newcommand{\lc}
{\mathrel{\raise2pt\hbox{${\mathop<\limits_{\raise1pt\hbox
{\mbox{$\sim$}}}}$}}}

\newcommand{\gc}
{\mathrel{\raise2pt\hbox{${\mathop>\limits_{\raise1pt\hbox{\mbox{$\sim$}}}}$}}}

\newcommand{\ec}
{\mathrel{\raise2pt\hbox{${\mathop=\limits_{\raise1pt\hbox{\mbox{$\sim$}}}}$}}}

\def\be{\begin{equation}} \def\ee{\end{equation}}

\def\bea{\begin{eqnarray}}  \def\eea{\end{eqnarray}}

\def\beas{\begin{eqnarray*}} \def\eeas{\end{eqnarray*}}

\def\bn{\begin{enumerate}} \def\en{\end{enumerate}}

\def\bd{\begin{description}} \def\ed{\end{description}}

\title{Reconstructing fine details of small objects by using plasmonic spectroscopic data}
\date{}

\author{
Habib Ammari\thanks{\footnotesize Department of Mathematics, 
ETH Z\"urich, 
R\"amistrasse 101, CH-8092 Z\"urich, Switzerland (habib.ammari@math.ethz.ch, sanghyeon.yu@math.ethz.ch). }
\and   Matias Ruiz\thanks{\footnotesize Department of Mathematics and Applications,
Ecole Normale Sup\'erieure, 45 Rue d'Ulm, 75005 Paris, France
(matias.ruiz@ens.fr).}
\and
Sanghyeon Yu \footnotemark[2]
\and
Hai Zhang\thanks{\footnotesize Department of Mathematics, HKUST, Clear Water Bay, Kowloon, Hong Kong (haizhang@ust.hk). The work of Hai Zhang was supported by HK RGC grant ECS 26301016 and startup fund R9355 from HKUST. }
}

\begin{document}
\maketitle

\begin{abstract}
This paper is concerned with the inverse problem of reconstructing a small object from far field measurements. The inverse problem is severally ill-posed because of the diffraction limit and low signal to noise ratio. We propose a novel methodology to solve this type of inverse problems based on an idea from plasmonic sensing. By using the field interaction with a known plasmonic particle, the fine detail information of the small object can be encoded into the shift of the resonant frequencies of the two particle system in the far field. In the intermediate interaction regime, we show that this information is exactly the generalized polarization tensors associated with the small object, from which one can perform the reconstruction. Our theoretical findings are supplemented by a variety of numerical results.  The results in the paper also provide a mathematical foundation for plasmonic sensing.   

\end{abstract}

\medskip

\bigskip

\noindent {\footnotesize Mathematics Subject Classification
(MSC2000): 35R30, 35C20.}


\section{Introduction}
The inverse problem of reconstructing fine details of small objects by using far-field measurements is severally ill-posed. There are two main reasons for this. The first reason is the diffraction limit. When illuminated by an incident wave with wavelength $\lambda$, the scattered field excited from the object which carries information on the scale smaller than $\lambda$ are confined near the object itself and only those with information on the scale greater than $\lambda$ can propagate into the far-field and be measured. As a result, from the far-field measurement one can only retrieve information about the object on the scale greater than $\lambda$. Especially in the case when the object is small (with a size smaller than $\lambda$), one can only obtain very few information. The second reason is the low signal to noise ratio. We know that small objects scatter "weakly".  This results in a very weak measurement signal in the far-field. In the presence of measurement noise, one has low signal to noise ratio and hence poor reconstruction. 

In this paper, we propose a new methodology to overcome the ill-posedness of this inverse problem. Our method is motivated by plasmonic bio-sensing. The key is to use a plasmonic particle to interact with the object and to propagate its near field information into far-field in term of shifts of plasmonic resonant frequencies. 

Plasmonic particles are metallic particles with size in the range from several nanometers to hundreds of nanometers. Under the illumination of electromagnetic field in the infrared and visible regime, the free-electrons in the particle may be strongly coupled to the electromagnetic field for certain frequencies resulting in strong scattering and enhancement of local fields. This phenomenon is called surface plasmon resonance \cite{SC10, Gri12} and the associated frequencies are called plasmonic resonant frequencies. Plasmonic resonance is extensively studied in the literature. A driving motivation is the use of plasmonic particles as the labels for sensing in molecular biology; see the review article \cite{anker} and the references therein. Besides sensing, there are other applications such as thermotherapy where plasmonic particles act as nanometric heat-generators that can be activated remotely by external electromagnetic fields \cite{baffou2010mapping}. We refer to \cite{SC10} and the references therein for these applications. We also refer to \cite{matias, pierre, kang1, hyeonbae, plasmon1, miller, plasmon4, yang}
for other related works of interest. 

The plasmon resonant frequency is one of the most important characterization of a plasmonic particle. It depends not only on the electromagnetic properties of the particle and its size and shape, but also on the electromagnetic properties of the environment \cite{matias, kelly, link}. It is the last property which enables the sensing application of plasmonic particles. Motivated by \cite{anker}, we perform in this paper a rigorous quantitative analysis for the sensing application. We show
that plasmonic resonance can be used to reconstruct fine details of small objects. We also remark that plasmonic resonance can also be used to identify the shape of the plasmonic particle itself \cite{yu}.

The methodology we propose is closely related to super-resolution in imaging. Super-resolution is about the separation of point sources.  In near field microscopy, the basic idea is to obtain the near field of sources which contains high resolution information.  This is made possible by propagating the near field information into the far field through certain near field interaction mechanism. In a recent series of papers \cite{hai, hai2, hai3}, we have shown mathematically how to use subwavelength resonators to achieve super-resolution. The idea is to obtain the near field information through the subwavelength resonant modes which can be excited by the sources with the right frequency and which can propagate into the far-field. In this paper,  we are interested in reconstructing the fine details of small objects in comparison to their positions and separability which are the focus of super-resolution.  The idea is similar. The near field information of the object is obtained from the near field interaction of the object and the plasmonic particle. 

In this paper, we consider the  system composed of a known plasmonic particle and the unknown object
whose geometry and electromagnetic properties are the quantities of interest. Under the illumination of incident waves with frequencies in certain range, we measure the frequencies where the peaks in the scattering field occur. These are the resonant frequencies or spectroscopic data of the system. By varying the relative position of the  particles, we obtain different resonant frequencies  due to the varying interactions between the  particles. We assume that the unknown particle is small compared to the plasmonic particle. In the intermediate regime when the distance of the two particles is comparable to the size of the plasmonic particle, we show that the presence of the small unknown particle can be viewed as a small perturbation to the homogeneous environment of the plasmonic particle. As a result, it induces a small shift to the plasmonic resonant frequencies of the plasmonic particle, which can be read from the observed spectroscopic data. By using rigorous asymptotic analysis, we obtain analytical formula for the shift which shows that the shift is determined by the generalized polarization tensors \cite{book2} of the unknown object. Therefore, from the far-field measurement of the shift of resonant frequencies, we can reconstruct the fine information of the object by using its generalized polarization tensors. 

We note that plasmonic resonant frequencies also depend on the size of the plasmonic particle \cite{matias, matias2, kelly, tocho}. In this paper, for the sake of simplicity, we consider the quasi-static approximation 
for the interaction between the electromagnetic field and the system of the two particles. Thus, we shall use the conductivity equation instead of the Helmholtz equation and the Maxwell equations. These more practical models will be analyzed in future works. In addition, we only consider the intermediate interaction regime in the paper, the strong interaction regime when the object is close to the plasmonic particle is also very interesting and will be reported in future works. 

This paper is organized in the following way. 
In Section \ref{sec-prelim}, we provide basic results on layer potentials and then explain the concept of plasmonic resonances and the (contracted) generalized polarization tensors.
In Section \ref{sec-forward}, we consider the forward scattering problem of the incident field interaction with a system composed of an ordinary particle and a plasmonic particle. We derive the asymptotic of the scattered field in the case of intermediate regime. In Section \ref{sec-inverse}, we consider the inverse problem of reconstructing the geometry of the ordinary particle. This is done by  constructing the generalized polarization tensors of the particles through the resonance shift  induced to the plasmonic particle.  In Section \ref{sec-numeric}, we provide numerical examples to justify our theoretical results. The paper ends with some concluding remarks.


\section{Preliminaries}\label{sec-prelim}

\subsection{Layer potentials and spectral theory of the NP operator}

We denote by $G(x,y)$ the Green function for the Laplacian in the free space. In $\R^2$, we have $$G(x,y) = \f{1}{2\pi}\log|x-y|.$$

Consider a domain $D$ with $\mathcal{C}^{1,\eta}$ boundary in $\mathbb{R}^2$ for $\eta>0$. Let $\nu(x)$ denote the outward normal at $ x \in \partial D$. Suppose that $D$ contains the origin $0$. 

The single layer potential $\mathcal{S}_{D}$ is 
given by 
$$
\mathcal{S}_{D} [\varphi](x) =\int_{\p D }  G(x,y)\varphi(y) d\sigma(y) ,   \quad x \in \mathbb{R}^2.
$$
The Neumann-Poincar\'{e} (NP) operator $\mathcal{K}_{D}^*$ associated with $D$  is defined as follows:
$$
\mathcal{K}_{D}^* [\varphi](x) = \int_{\p D }  \frac{ \p G }{\p\nu(x)} (x,y) \varphi(y) d\sigma(y) ,   \quad x \in \p D.
$$
The following jump relations hold:
\begin{align}
{\mathcal{S}_D[\varphi]}\big|_+ &= {\mathcal{S}_D[\varphi]}\big|_-,\label{eqn_jump_single1}
\\
\frac{\p\mathcal{S}_D[\varphi]}{\p\nu}\Big|_{\pm} &= (\pm\frac{1}{2} I +\mathcal{K}_D^*)[\varphi]. \label{eqn_jump_single2}
\end{align}

Let $H^{1/2}(\p D)$ be the usual Sobolev space and let $H^{-1/2}(\p D)$ be its dual space with respect to the $L^2$-pairing $(\cdot,\cdot)_{-\frac{1}{2},\frac{1}{2}}$. We denote by $H^{-1/2}_0(\p D)$ the collection of all $\varphi\in H^{-1/2}(\p D)$ such that $(\varphi,1)_{-\frac{1}{2},\frac{1}{2}}=0$.

The NP operator is bounded on $H^{-1/2}(\p D)$ and maps $H^{-1/2}(\p D)$ into itself.
It can be shown that the operator $\lambda I - \mathcal{K}_D^*: L^2(\partial D)
\rightarrow L^2(\partial D)$ is invertible for any $|\lambda| > 1/2$.
Although the NP operator is not self-adjoint on $L^2(\p D)$, it can be symmetrized on $H_0^{-1/2}(\p D)$ by using a new inner product. 
Let $\mathcal{H}^*(\p D)$ be the space $H^{-1/2}_0(\p D)$ equipped with the inner product $(\cdot,\cdot)_{\mathcal{H}^*(\p D)}$ defined by
$$
(\varphi,\psi)_{\mathcal{H}^*(\p D)} =  -(\varphi,\mathcal{S}_D[\psi])_{-\frac{1}{2},\frac{1}{2}},
$$
for $\varphi,\psi\in H^{-1/2}(\p D)$.
Then using the Plemelj's symmetrization principle,
$$
\mathcal{S}_D\mathcal{K}_D^*=\mathcal{K}_D\mathcal{S}_D,
$$
it can be shown that the NP operator $\mathcal{K}_D^*$ is self-adjoint with respect to $(\cdot,\cdot)_{\mathcal{H}^*(\p D)}$.
 Furthermore, $\mathcal{K}_D^*$ is compact, so its spectrum is discrete and contained in $]-1/2, 1/2]$; see for instance \cite{book2} for more details. 
Therefore, the NP operator $\mathcal{K}_D^*$ admits the following spectral decomposition:
for $\varphi\in\mathcal{H}^*$,
\be\label{spectral_decomposition_Kstar}
\mathcal{K}_D^*[\varphi] = \sum_{j=1}^\infty \lambda_j (\varphi,\varphi_j)_{\mathcal{H}^*} \varphi_j,
\ee 
where $\lambda_j$ are the eigenvalues of $\mathcal{K}_D^*$ and $\varphi_j$ are their associated eigenfunctions. Note that $|\lambda_j|<1/2$ for all $j\geq 1$.

\subsection{Plasmonic resonance}\label{subsec-plasmonic}

We are interested in the regime when a plasmonic resonance  occurs,  so the wavelength of the incident field should be much greater than the size of the plasmonic particle.  To further simplify the analysis and better illustrate the main idea of our methodology, we use the quasi-static approximation (by assuming the incident wavelength to be infinity) to model the interaction.

Given a harmonic function $H$ in $\mathbb{R}^2$, which represents an incident field, we consider the following transmission problem:
\be
    \begin{cases}
        \nabla \cdot (\eps \nabla u) = 0 &\text{ in }\; \mathbb{R}^2, \\[1.5mm]
         u - H = O(|x|^{-1}) &\text{ as }\; |x| \rightarrow \infty,
    \end{cases}
    \label{transmission}
\ee
where $\eps  = \eps_D \chi(D) +  \eps_m\chi(\mathbb{R}^2 \backslash \overline{D})$, and $\chi(D)$ and $\chi(\mathbb{R}^2 \backslash \overline{D})$ are  the characteristic functions of $D$ and $\mathbb{R}^2 \backslash \overline{D}$, respectively.  From \cite{book2}, we have
\be
    u = H+ \mathcal{S}_{D} [\varphi]  \, ,
    \label{scattered}
\ee
where $\varphi$ satisfies
\be\label{phi_int_equation}
 (\lambda I - \mathcal{K}_D^*)[\varphi] = \frac{\partial H}{\partial \nu}\Big|_{\p D}.
\ee
Here, $\lambda$ is given by
 \be \label{deflambda} 
\lambda= \frac{\eps_D+\eps_m}{2(\eps_D-\eps_m)}.\ee  

Contrary to ordinary dielectric particles, the permittivities of plasmonic materials, such as noble metals, have negative real parts. In fact, the permittivity $\eps_D$ depends on the operating frequency $\omega$ and can be modeled by the Drude's model given by
\be
\eps_D = \eps_D(\omega) = 1-\frac{\omega_p^2}{\omega(\omega+i\gamma)},
\ee
where $\omega_p>0$ is called the plasma frequency and $\gamma>0$ is the damping parameter. Since the parameter $\gamma$ is typically very small,
the permittivity $\eps_D(\omega)$ has a small imaginary part.

Now we discuss the plasmonic resonances.
By applying the spectral decomposition \eqref{spectral_decomposition_Kstar} of $\mathcal{K}_D^*$ to the integral equation \eqref{phi_int_equation}, the density $\varphi$ becomes
\be\label{varphi_spectral_decomposition}
\varphi = \sum_{j=1}^\infty \frac{(\frac{\p H}{\p\nu},\varphi_j)_{\mathcal{H}^*(\p D)}}{\lambda_D- \lambda_j} \varphi_j.
\ee
Recall that $\lambda_j$ are eigenvalues $\mathcal{K}_D^*$ and they satisfy $|\lambda_j|<1/2$. 
For $\omega<\omega_p$, $\mbox{Re}\{\eps_D(\omega)\}$ can take  negative values. Then it holds that $|\mbox{Re}\{\lambda(\omega)\}|<1/2$. So, for a certain frequency $\omega_j$, the value of $\lambda(\omega_j)$ can be very close to an eigenvalue $\lambda_j$ of the NP operator. Then, in \eqref{varphi_spectral_decomposition}, the eigenfunction $\varphi_j$ will be amplified provided that $(\frac{\p H}{\p\nu},\varphi_j)_{\mathcal{H}^*(\p D)}$ is non-zero. As a result, the scattered field $u-u^i$ will show a resonant behavior. This phenomenon is called the plasmonic resonance.

When $D$ is an ellipse, we can compute the spectral properties of the NP operator $\mathcal{K}_D^*$ explicitly.
Let $D$ be an ellipse given by
\be\label{def_ellipse}
D = \Big\{(x,y)\in\mathbb{R}^2: \frac{x^2}{a^2} + \frac{y^2}{b^2} \leq 1\Big\},
\ee
for some constants $a,b$ with $a<b$.
 Then it is known that the eigenvalues of the NP operator associated with the ellipse $D$ on $\mathcal{H}^*$ are 
$$
\pm \frac{1}{2}\Big(\frac{b-a}{b+a}\Big)^{j}, \quad j=1,2,3,\cdots.
$$

\subsection{Contracted generalized polarization tensors}\label{subsec-CGPT}

Here we explain the concept of the generalized polarization tensors (GPTs).
The scattered field $u-u^i$ has the following far-field behavior  \cite[p. 77] {book2}
\be
    (u - u^i)(x) =  \sum_{|\alpha|, |\beta| \geq 1 } \f{1}{\alpha ! \beta!} \partial^\alpha u^i(0) M_{\alpha \beta}(\lambda, D) \partial^\beta G(x), 
    \quad |x| \rightarrow + \infty,
    \label{scattered2}
\ee
where $M_{\alpha \beta}(\lambda,D)$ is given by
$$M_{\alpha \beta}(\lambda, D) : =  \int_{\partial D} y^\beta (\lambda I - \mathcal{K}_D^*)^{-1} [\frac{\partial x^\alpha}{\partial \nu}](y)\, d\sigma(y), \qquad \alpha, \beta \in \mathbb{N}^d.$$
Here, the coefficient $M_{\alpha \beta}(\lambda, D)$ is called the {\it generalized polarization tensor} \cite{book2}.

For a positive integer
$m$, let $P_m(x)$ be the complex-valued polynomial
\begin{equation}
P_m(x) = (x_1 + ix_2)^m = r^m \cos m\theta +ir^m \sin m\theta, \label{eq:Pdef}
\end{equation}
where we have used the polar coordinates $x = re^{i\theta}$.

We define the {\it contracted generalized polarization tensors }(CGPTs) to be the
following linear combinations of generalized polarization tensors using the polynomials in
\eqref{eq:Pdef}:
\begin{align*}
M^{cc}_{m,n}(\lambda,D) = \int_{\partial D} \mbox{Re} \{ P_n\}  (\lambda I - \mathcal{K}_D^*)^{-1} [\frac{\partial \,\mbox{Re} \{ P_m\}}{\partial \nu}]\, d\sigma,  \\
M^{cs}_{m,n}(\lambda,D) = \int_{\partial D} \mbox{Im} \{ P_n\}  (\lambda I - \mathcal{K}_D^*)^{-1} [\frac{\partial \,\mbox{Re} \{ P_m\}}{\partial \nu}]\, d\sigma,\\
M^{sc}_{m,n}(\lambda,D) = \int_{\partial D} \mbox{Re} \{ P_n\}  (\lambda I - \mathcal{K}_D^*)^{-1} [\frac{\partial \,\mbox{Im} \{ P_m\}}{\partial \nu}]\, d\sigma,\\
M^{ss}_{m,n}(\lambda,D) = \int_{\partial D} \mbox{Im} \{ P_n\}  (\lambda I - \mathcal{K}_D^*)^{-1} [\frac{\partial \,\mbox{Im} \{ P_m\}}{\partial \nu}]\, d\sigma. 
\end{align*}
We refer to \cite{book2} for further details.

For convenience, we introduce the following notation. We denote
\beas
M_{m,n}(\lambda,D) &=& \left( \begin{array}{c c}
M_{m,n}^{cc}(\lambda,D) & M_{m,n}^{cs}(\lambda,D)\\
M_{m,n}^{sc}(\lambda,D) & M_{m,n}^{ss}(\lambda,D)
\end{array}\right).
\eeas

When $m=n=1$, the matrix $M(\lambda,D):=M_{1,1}(\lambda,D)$ is called the {\it first order polarization tensor}.
Specifically, we have
$$
M(\lambda,D)_{lm} =  \int_{\partial D} y_j (\lambda I - \mathcal{K}_D^*)^{-1} [\nu_i](y)\, d\sigma(y), \quad l,m=1,2.
$$
Since, from \eqref{scattered2}, we have
$$
(u-u^i)(x) = \frac{\nabla u^i \cdot M(\lambda,D) x}{|x|^2} + O(|x|^{-2}), \quad
\mbox{as }|x|\rightarrow \infty,
$$
the first order polarization tensor $M(\lambda,D)$ determines the dominant term in the far-field expansion of the scattered field $u-u^i$.

To see the plasmonic resonant behavior of the far field, it is helpful to represent $M(\lambda,D)$ in a spectral form.
By the spectral decomposition \eqref{spectral_decomposition_Kstar}, we obtain that
$$
M(\lambda,D)_{lm} = \sum_{j=1}^\infty \frac{(y_m,\varphi_j)_{-\frac{1}{2},\frac{1}{2}}(\varphi_j,\nu_l)_{\mathcal{H}^*(\p D)}}{\lambda-\lambda_j}.
$$
If $D$ is the ellipse given by \eqref{def_ellipse}, then we have the explicit formula for $M(\lambda,D)$ 
\begin{align} \label{mellips}
M(\lambda,D)=
\begin{pmatrix}
 \frac{\pi a b}{ \lambda - \frac{1}{2}\frac{a-b}{a+b} } & 0
 \\
 0 & \frac{\pi a b}{ \lambda + \frac{1}{2}\frac{a-b}{a+b} }
\end{pmatrix}.
\end{align}
Formula \eqref{mellips} indicates that, in the far field region, the plasmonic resonance occurs only if $\lambda$ is close to $\frac{1}{2}\frac{a-b}{a+b}$ or $-\frac{1}{2}\frac{a-b}{a+b}$.


\section{The forward problem} \label{sec-forward}
We consider a system composed of a small ordinary particle and a plasmonic particle embedded in a homogeneous medium; see Figure \ref{fig-two_particles sensing}.  
The ordinary particle and the plasmonic particle occupy a bounded and simply connected domain $D_1\subset\mathbb{R}^2$ and $D_2\subset\mathbb{R}^2$ of class $\mathcal{C}^{1,\alpha}$ for some $0<\alpha<1$, respectively. We denote the permittivity of the ordinary particle $D_1$ (or the plasmonic particle $D_2$) by $\eps_1$ (or $\eps_2$), respectively.
The permittivity of the background medium is denoted by $\varepsilon_m$.
In other words, the permittivity distribution $\eps$ is given by
$$
\eps:=
\eps_1\chi(D_1)  + \eps_2\chi(D_2) +  \eps_m\chi(\R^2\backslash( \overline{D_1\cup D_2})).
$$
As in Subsection \ref{subsec-plasmonic}, the permittivity $\eps_2$ of the plasmonic particle depends on the operating frequency and is modeled as
$$
\eps_2=\eps_2(\omega) = 1-\frac{\omega_p^2}{\omega(\omega+i\gamma)}.
$$
We assume the following condition on the size of the particles $D_1$ and $D_2$.
\begin{cond} \label{condition0 biosensing} 
 The plasmonic particle $D_2$ has size of order one and is centered at a position that we denote by $z$;  the ordinary particle $D_1$ has size of order $\delta \ll 1$ and is centered at the origin. Specifically, we write $D_1= \delta B$, where the domain $B$ has size of order one.  
\end{cond}

The total electric potential $u$ satisfies the following equation:
\be \label{eq-Helmholtz eq biosensing}
\begin{cases}
	\ds \nabla \cdot (\eps \nabla u)  = 0 &\quad \mbox{in } \R^2\backslash (\p D_1\cup \p D_2), \\
	\nm
	 u|_{+} = u|_{-}    &\quad \mbox{on } \partial D_1 \cup \partial D_2, \\
	\nm
	  \ds \eps_{m} \df{\p u}{\p \nu} \Big|_{+} =\eps_{1} \df{\p u}{\p \nu} \Big|_{-}  &\quad \mbox{on } \partial D_1,\\
	\nm
	  \ds \eps_{m} \df{\p u}{\p \nu} \Big|_{+} = \eps_{2} \df{\p u}{\p \nu} \Big|_{-}  &\quad \mbox{on } \partial D_2,\\
	\nm
	  (u-u^i)(x) = O(|x|^{-1}),   &\quad\mbox{as }|x|\rightarrow \infty,
\end{cases}
\ee
where $u^i(x)=d\cdot x$ is the incident potential with a constant vector $d\in\mathbb{R}^2$.


\begin{figure}[h!]
\begin{center}
\includegraphics[scale=0.7]{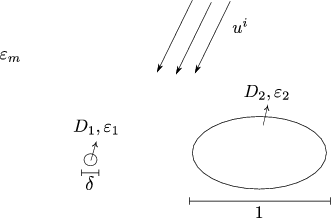}
\caption{ \label{fig-two_particles sensing} Scattering of an incident wave $u^{i}$ by a system of a plasmonic ($D_2$) - non plasmonic ($D_1$) particles.}
\end{center}
\end{figure}

\subsection{The Green function in the presence of a  small particle}

Let $G_{D_1}(\cdot,y)$ be the Green function  at the source point $y$  of a medium consisting of the particle $D_1$, which is embedded in the free space. For every $y\notin \overline{D_1}$, $G_{D_1}(\cdot,y)$ satisfies the following equation:
\be \label{eq-Helmholtz eq green function biosensing}
\begin{cases}
 	\ds \nabla \cdot \big(\eps_1\chi(D_1) +  \eps_m\chi(\R^2\backslash \overline{D_1})\big) \nabla u  = \delta_y &\quad \mbox{in } \R^2\backslash \p D_1, \\
 	\nm
 	 u|_{+} =u|_{-}     &\quad \mbox{on } \partial D_1, \\
 	\nm
 	  \ds \eps_{m} \df{\p u}{\p \nu} \bigg|_{+} = \eps_{1} \df{\p u}{\p \nu} \bigg|_{-}  &\quad \mbox{on } \partial D_1,\\
 	\nm
   u(x) = O(|x|^{-1}),   &\quad\mbox{as }|x|\rightarrow \infty.
\end{cases}
\ee

We look for a solution of the form:

\be \label{eq-representation green function sensing}
G_{D_1}(x,y) :=
\begin{array}{cc}
	G(x,y) + \mathcal{S}_{D_1} [\psi], & \quad x \in \R^2 \backslash \overline{D_1}
\end{array}.
\ee
Note that $G_{D_1}$ satisfies the second and fourth conditions in \eqref{eq-Helmholtz eq green function biosensing}.
From the third condition in \eqref{eq-Helmholtz eq green function biosensing} and the jump formula \eqref{eqn_jump_single2} for the single layer potential, the density $\psi$ must satisfy the following equation on $\p D_1$:
\be \label{Helm-syst green fcn biosensing}
\begin{array}{rcl}
	\eps_m\big(\df{1}{2}Id + \mathcal{K}_{D_1}^*\big)[\psi] +  \eps_1 \big(\df{1}{2}Id- \mathcal{K}_{D_1}^*\big)[\psi] = (\eps_1-\eps_m)\df{\p}{\p \nu}G(\cdot,y).
\end{array}
\ee
So we obtain
\beas
\psi &=& \left(\lambda_{D_1}Id - \mathcal{K}_{D_1}^*\right)^{-1}\Big[\df{\p }{\p \nu}G(\cdot,y)\Big],
\\
\lambda_{D_1} &=& \df{\eps_1+\eps_m}{2(\eps_1-\eps_m)}.
\eeas
Therefore, from \eqref{eq-representation green function sensing} and the uniqueness of a solution to \eqref{eq-Helmholtz eq green function biosensing}, we have the following representation for the Green's function $G_{D_1}$:
\be \label{eq-green function D_1 sensing}
G_{D_1}(x,y) = G(x,y) + \mathcal{S}_{D_1}\left(\lambda_{D_1}Id - \mathcal{K}_{D_1}^*\right)^{-1}\Big[\df{\p }{\p \nu}G(\cdot,y)\Big](x) \quad\mbox{for } x,y \in \R^2 \backslash \overline{D_1}.
\ee

\subsection{Representation of the total potential}

Here we derive a layer potential representation of the total potential $u$, which is the solution to \eqref{eq-Helmholtz eq biosensing}.

Let 
$u_{D_1}$ be the total field resulting from the incident field $u^i$ and the ordinary particle $D_1$ 
(without the plasmonic particle $D_2$).
Note that $u_{D_1}$ is given by
\beas
u_{D_1}(x) = u^i(x) + \mathcal{S}_{D_1}\left(\lambda_{D_1}Id - \mathcal{K}_{D_1}^*\right)^{-1}[\df{\p u^i}{\p \nu_1}](x), \quad \mbox{ for }  x \in \R^2 \backslash \overline{D_1}.
\eeas
To consider 
the total potential $u$,
we also need to represent the field generated by the plasmonic particle $D_2$. For this, we introduce a new layer potential $\mathcal{S}_{D_2,D_1}$ as follows:
\beas
\mathcal{S}_{D_2,D_1}[\varphi](x) = \int_{\p D_2}G_{D_1}(x,y)\varphi(y)d\sigma(y).
\eeas

The total potential $u$ can be represented in the following form:
\be \label{solution helm nanoparticle biosenging}
u = u_{D_1} + \mathcal{S}_{D_2,D_1} [\psi], \quad x \in \R^2 \backslash \overline{D_2}.
\ee
We need to find a boundary integral equation for the density $\psi$.
It follows from \eqref{eq-green function D_1 sensing} that, for any $\varphi$,
\beas
\mathcal{S}_{D_2,D_1}[\varphi](x) = \mathcal{S}_{D_2}[\varphi](x)+\mathcal{S}_{D_2,D_1}^{1}[\varphi](x),
\eeas
where $\mathcal{S}_{D_2,D_1}^{1}$ is given by
\beas
\mathcal{S}_{D_2,D_1}^{1}[\varphi](x) := \int_{\p D_2}\mathcal{S}_{D_1}\left(\lambda_{D_1}Id - \mathcal{K}_{D_1}^*\right)^{-1}[\df{\p }{\p \nu_1}G(\cdot,y)](x)\varphi(y)d\sigma(y).
\eeas
The expression of $\mathcal{S}_{D_2,D_1}^{1}[\varphi]$ can be further developed using the following spectral expansion of the free-space Green function $G$ \cite{kang1}:
\beas
G(x,y) = -\sum_{j=0}^{\infty}\mathcal{S}_{D}[\varphi_j](x)\mathcal{S}_{D}[\varphi_j](y) + \mathcal{S}_D[\varphi_0](x),\quad\mbox{for }  x \in \R^2 \backslash \overline{D} \mbox{ and }  y\in \overline{D},
\eeas
where $\varphi_j,j=1,2,...$ are eigenfunctions of $\mathcal{K}_D^*$ on $\mathcal{H}^*(\p D)$ and $\varphi_0$ is an eigenfunction associated to the eigenvalue $1/2$.
Then, for any $\varphi\in\mathcal{H}^*$, we get
\begin{align*}
\int_{\p D_2} G(\cdot,y) \varphi(y)d\sigma(y) &= 
\sum_{j=1}^\infty\mathcal{S}_{D_2}[\varphi_j]( \varphi,\varphi_j)_{\mathcal{H}^*(\p D_2)} + \mathcal{S}_D[\varphi_0](x)\int_{\p D_2} \varphi(y) 
\\
&= \sum_{j=1}^\infty\mathcal{S}_{D_2}[\varphi_j] ( \varphi,\varphi_j)_{\mathcal{H}^*(\p D_2)}.
\end{align*}
Therefore, for any $\varphi \in\mathcal{H}^*$, we have,
\beas
\mathcal{S}_{D_2,D_1}^{1}[\varphi](x) &=& \int_{\p D_2}\mathcal{S}_{D_1}\left(\lambda_{D_1}Id - \mathcal{K}_{D_1}^*\right)^{-1}[\df{\p }{\p \nu_1}G(\cdot,y)](x)\varphi(y)d\sigma(y)
\\
&=&  \mathcal{S}_{D_1}\left(\lambda_{D_1}Id - \mathcal{K}_{D_1}^*\right)^{-1}\df{\p }{\p \nu_1}\mathcal{S}_{D_2}\Big[\sum_{j=0}^{\infty}(\varphi,\varphi_j)_{\mathcal{H}^*}\varphi_j\Big] (x)\\
&=&  \mathcal{S}_{D_1}\left(\lambda_{D_1}Id - \mathcal{K}_{D_1}^*\right)^{-1}\df{\p \mathcal{S}_{D_2}[\varphi]}{\p \nu_1}(x),
\eeas
where we have used the notation $\f{\p}{\p \nu_i}$ to indicate the outward normal derivative on $\p D_i$. 
 
Combining the boundary conditions in \eqref{eq-Helmholtz eq biosensing}, the representation formula \eqref{solution helm nanoparticle biosenging} and the jump formula \eqref{eqn_jump_single2} yields the following equation for $\psi$
\beas
\left(\mathcal{A}_{D_2,0} + \mathcal{A}_{D_2,1}\right)[\psi] = \df{\p u_{D_1}}{\p \nu_2},
\eeas
where
\bea
\mathcal{A}_{D_2,0} &=& \lambda_{D_2}Id - \mathcal{K}_{D_2}^*,\nonumber\\
\lambda_{D_2} &=& \df{\eps_2+\eps_m}{2(\eps_2-\eps_m)}, \label{def-lambda_2 biosensing}\\
\mathcal{A}_{D_2,1} &=& \df{\p \mathcal{S}_{D_2,D_1}^{1}}{\p \nu_2} \; =\: \f{\p }{\p \nu_2}\mathcal{S}_{D_1}\left(\lambda_{D_1}Id - \mathcal{K}_{D_1}^*\right)^{-1}\df{\p \mathcal{S}_{D_2}[\cdot]}{\p \nu_1}. \label{def-perturbation operator biosensing}
\eea

\subsection{Intermediate regime and asymptotic expansion of the scattered field}

Here
we introduce the concept of intermediate regime and derive the asymptotic expansion of the scattered field $u-u^i$ for small $\delta$.

\begin{definition}[\textbf{Intermediate regime}] \label{def-intermediate sensing}
We say that $D_2$ is in the intermediate regime with respect to the origin if there exist positive constants $C_1$ and $C_2$ such that $C_1 < C_2$ and
$$
C_1 \leq 
{\rm{dist}} (0,D_2) \leq C_2.
$$
\end{definition}

Definition \ref{def-intermediate sensing} says that the plasmonic particle $D_2$ is located not too close to $D_1$ nor far from $D_1$.
Throughout this paper, we assume the plasmonic particle $D_2$ is in the intermediate regime.
We have the following result.

\begin{prop} \label{prop-intermediate-AD21-small}
If $D_2$ is in the intermediate regime, then
$\|\mathcal{A}_{D_2,1}\|_{\mathcal{H}*}=O(\delta^2)$ as $\delta\rightarrow 0$.
\end{prop}
\proof
Fix  $\varphi\in\mathcal{H}^*(\p D_2)$ and 
let
$$\widetilde{\varphi}:=(\lambda_{D_1}Id -\mathcal{K}_{D_1}^*)^{-1}\Big[\frac{\p \mathcal{S}_{D_2}[\varphi]}{\p\nu_1} \Big].$$
Since $\mathcal{S}_{D_2}[\varphi]$ is harmonic in $D_1$, the Green's identity gives
$
\int_{\p D_1}\frac{\p}{\p\nu_1}{ \mathcal{S}_{D_2}[\varphi]} = 0.
$
Then it can be proved that $\int_{\p D_1}\widetilde{\varphi}=0$.
So we get
\begin{align*}
\mathcal{S}_{D_1}[\widetilde{\varphi}](x)
&=\int_{\p D_1} (\log|x-y| - \log|x|) \widetilde{\varphi}(y) d\sigma(y) + \log|x|\int_{\p D_1} \widetilde{\varphi}(y) d\sigma(y)
\\
&=\int_{\p D_1} (\log|x-y| - \log|x|) \widetilde{\varphi}(y) d\sigma(y). 
\end{align*}
Therefore, since $|y-x|\geq C'$ and $|y|\leq C \delta$ for $(y,x)\in (\p D_1,  \p D_2)$, we obtain 
$$
\| \mathcal{A}_{D_2,1}[{\varphi}] \|_{\mathcal{H}^*(\p D_2)}
=
\big\| \frac{\p}{\p\nu_2}\mathcal{S}_{D_1}[\widetilde{\varphi}] \big\|_{\mathcal{H}^*(\p D_2)}
\leq
C \delta \| \widetilde{\varphi}\|_{\mathcal{H}^*(\p D_1)}.
$$

Now it suffices to prove that
\begin{equation}
\| \widetilde\varphi\|_{\mathcal{H}^*(\p D_1)} \leq C \delta.
\end{equation}
Recall that $D_1 = \delta B$.
Let $f_\delta(y)=f(\delta y)$. Then the function $f_\delta$ belongs to $\mathcal{H}^*(\p B)$ for $f \in \mathcal{H}^*(\p D_1)$. 
Since it is known that $\mathcal{K}_\Omega^*$ is scale-invariant for any $\Omega$, we have $\mathcal{K}_{D_1}^*[f] =\mathcal{K}_{B}^*[f_\delta]$. Therefore,
$$
\widetilde{\varphi}= \left(\lambda_{D_1}Id - \mathcal{K}_{D_1}^*\right)^{-1}[f] d(\delta\sigma(y))   = 
\left(\lambda_{D_1}Id - \mathcal{K}_{B}^*\right)^{-1}[f_\delta] d(\delta\sigma(y))  . 
$$
Again, since 
 $|y-x|\geq C'$ for $(y,x)\in (\p D_1,  \p D_2)$ and $|\p D_1|= O(\delta)$, we arrive at
\begin{align*}
\| \widetilde\varphi\|_{\mathcal{H}^*(\p D_1)}
&=
\|\left(\lambda_{D_1}Id - \mathcal{K}_{B}^*\right)^{-1}
\Big[\Big(\frac{\p \mathcal{S}_{D_2}[\varphi]}{\p\nu_1}\Big)_\delta \Big]
\|_{\mathcal{H}^*(\p B)}
\\
&\leq C\|\frac{\p }{\p\nu_1}\mathcal{S}_{D_2}[\varphi]
\|_{\mathcal{H}^*(\p D_1)} \leq C \delta.
\end{align*}
The proof is completed.
\qed


\smallskip

From Proposition \ref{prop-intermediate-AD21-small}, we can view $\mathcal{A}_{D_2,1}$
as a perturbation of $\mathcal{A}_{D_2,0}$.
Using standard perturbation theory \cite{berry}, we can derive the perturbed eigenvalues and associated eigenfunctions.

Let $\lambda_j$ and $\varphi_j$ be the eigenvalues and eigenfunctions of $\mathcal{K}_{D_2}^*$ on $\mathcal{H}^*(\p D_2)$.
For simplicity, we consider the case when $\lambda_j$ is a  simple eigenvalue of the operator $\mathcal{K}_{D_2}^*$.
Let us define
\be
R_{jl} = \big( {\mathcal{A}}_{D_2,1}[\varphi_l], \varphi_j \big)_{\mathcal{H^*}(\p D_2)},
\ee
where ${\mathcal{A}}_{D_2, 1}$ is given by \eqref{def-perturbation operator biosensing}. Note that $R_{jl} = O(\delta^2)$.

The perturbed eigenvalues have the following form:
\beas
\tau_j (\delta) = \lambda_{D_2}-\lambda_j + \mathcal{P}_{j}, \label{tau-single}
\eeas
where $\mathcal{P}_j$ are given by
\begin{align}
\mathcal{P}_{j} &= R_{jj} + \sum_{l\neq j}\f{R_{jl}R_{lj}}{\lambda_j-\lambda_l} + \sum_{(l_1,l_2)\neq j}\f{R_{jl_2}R_{l_2l_1}R_{l_1j}}{(\lambda_j-\lambda_{l_1})(\lambda_j-\lambda_{l_2})} 
\nonumber
\\
&\quad
+ \sum_{(l_1,l_2,l_3)\neq j}\f{R_{jl_3}R_{l_3l_2}R_{l_2l_1}R_{l_1j}}{(\lambda_j-\lambda_{l_1})(\lambda_j-\lambda_{l_2})(\lambda_j-\lambda_{l_3})} + \cdots.
\label{Pj_expansion}
\end{align}
Also, the perturbed eigenfunctions have the following form:
\bea
\varphi_j(\delta) = \varphi_j + O(\delta^2). \label{eigenfun-single}
\eea
Here the remainder term is with respect to the norm $\| \cdot\|_{\mathcal{H}^*(\p D_2)}$.

\begin{rmk} Note that $\mathcal{P}_j$ depends not only on the geometry and material properties of $D_1$, but also on $D_2$'s properties, in particular its position $z$.
\end{rmk}

\begin{thm} \label{thm1 biosensing}
If $D_2$ is in the intermediate regime, the scattered field $u^s_{D_2}=u- u_{D_1}$ by the plasmonic particle $D_2$ has the following  representation:
$$
u^s_{D_2} = \mathcal{S}_{D_2,D_1} [\psi],
$$
where $\psi$ satisfies
\beas
\psi = \sum_{j=1}^{\infty}\f{ \left(\nabla u^i(z)\cdot\nu,\varphi_j\right)_{\mathcal{H}^*(\p D_2)} \varphi_j + O(\delta^2)}{ \lambda_{D_2} - \lambda_j + \mathcal{P}_j}
\eeas
with $\lambda_{D_2}$ being given by (\ref{def-lambda_2 biosensing}). 
\end{thm}

As a corollary, we have the following asymptotic expansion of the scattered field $u-u^i$.
\begin{thm} \label{thm-far field sensing}
We have the following far field expantion:
\beas
(u-u^i)(x) =   \nabla u^i(z)\cdot M(\lambda_{D_1},\lambda_{D_2},D_1,D_2) \nabla G(x,z) + O(\delta^2) +O\left(\f{\delta^3 }{\textnormal{dist}(\lambda_{D_2},\sigma(\mathcal{K}^*_{D_2}))}\right),
\eeas
as $|x|\rightarrow \infty$.
Here, $M(\lambda_{D_1},\lambda_{D_2},D_1,D_2)$ is the polarization tensor satisfying
\be \label{eq-PT_D1D2 sensing}
M(\lambda_{D_1},\lambda_{D_2},D_1,D_2)_{l,m} = \sum_{j=1}^{\infty}\f{(\nu_l,\varphi_j)_{\mathcal{H}^*(\p D_2) } (\varphi_j,x_m)_{-\f{1}{2},\f{1}{2}} + O(\delta^2)}{ \lambda_{D_2} - \lambda_j +  \mathcal{P}_j },
\ee
for $l,m=1,2$.
\end{thm}
We remark that the scattered field in the above expression depends on the frequency (since $\lambda_{D_2}$ does so) and exhibit local peaks at certain frequencies when one of the denominators is close to zero and is minimized while the associated nominator is not zero. These frequencies are called the resonant frequencies of the system. It is clear that these resonant frequencies also depend on the geometry and the electric permittivity of $D_1$ through the perturbative terms $\mathcal{P}_j$'s. We shall use this fact in the next section to solve the associated inverse problem of reconstructing $D_1$ by using those frequencies. 

\subsection{Representation of the shift $\mathcal{P}_j$ using CGPTs}

Here we show that the term $\mathcal{P}_j$ in the plasmonic resonances can be expressed in terms of the CGPTs.  
The CGPTs carry information on the geometry and material properties of $D_1$. 
See \cite{book2} for a detailed reference. We shall reconstruct the ordinary particle $D_1$ from the measurement of the shift $\mathcal{P}_j$.

\begin{prop}  \label{prop-GPT perturbation}
If $D_2$ is in the intermediate regime, then the perturbative terms $R_{jl}$ can be represented using CGPTs $M_{m,n}(\lambda_{D_1},D_1)$ associated with $D_1$ as follows: 
\be\label{Rjl_CGPT}
R_{jl} = \left(\f{1}{2}-\lambda_j\right)\sum_{m= 1}^{M} \sum_{ n = 1}^{N} a_{m}^j M_{m,n}(\lambda_{D_1},D_1)(a_{n}^l)^t+ O(\delta^{M+N+1}),
\ee
where the superscript $t$ denotes the transpose and $a_{m}^j  = (a_{m,c}^j,a_{m,s}^j)$ with 
\beas
a_{m,c}^j &=&  -\f{1}{2\pi m}\int_{\p D_2} \f{\cos(m \theta_y)}{ r_y^m} \varphi_j(y)d\sigma(y) ,\\
a_{m,s}^j &=&  -\f{1}{2\pi m}\int_{\p D_2}\f{\sin(m \theta_y)}{ r_y^m} \varphi_j(y)d\sigma(y).
\eeas
Here, $(r_y,\theta_y)$ denote the polar coordinates of $y$ and $\{\varphi_j\}_j$ is an orthonormal basis of eigenfunctions of $\mathcal{K}_{D_2}^*$ on $\mathcal{H}^*$.
\end{prop}

\begin{proof}
To simplify the notation, let us denote
$$
F_l=\mathcal{S}_{D_1}\left(\lambda_{D_1}Id - \mathcal{K}_{D_1}^*\right)^{-1}\df{\p \mathcal{S}_{D_2}[\varphi_l]}{\p \nu_1}.
$$
Then, from the Green's identity and the jump formula \eqref{eqn_jump_single2}, we obtain
\beas
R_{jl} &=& \big( F_l, \varphi_j \big)_{\mathcal{H^*}}
= -\big(\f{\p F_l}{\p \nu_2}, \mathcal{S}_{D_2}[\varphi_j] \big)_{\f{1}{2},-\f{1}{2}}
\\
&=& - \big(F_l, \f{\p \mathcal{S}_{D_2}[\varphi_j] }{\p \nu_2}\bigg|_{-}\big)_{\f{1}{2},-\f{1}{2}}
=- \big(F_l, (-\frac{1}{2} + \mathcal{K}_{D_2}^*)[\varphi_j]\big)_{\f{1}{2},-\f{1}{2}}.
\eeas
Since $\varphi_j$ is an eigenfunction of $\mathcal{K}_{D_2}^*$ with an eigenvalue $\lambda_j$, we have 
\beas
R_{jl}&=& \Big(\f{1}{2}-\lambda_j\Big) \big(F_l, \varphi_j \big)_{\f{1}{2},-\f{1}{2}}.
\eeas

Let $(r_x,\theta_x)$ be the polar coordinates of $x$. 
It is known from \cite{book3} that, for $|x|<|y|$, 
\be\label{eq-green GPT expansion1}
G({x},y) = \sum_{n=0}^{\infty} \f{(-1)}{2\pi n}  \f{\cos(n \theta_y)}{ r_y^n} r_{{x}}^n\cos(n \theta_{{x}}) + \f{(-1)}{2\pi n} \f{\sin(n \theta_y)}{ r_y^n} r_{{x}}^n\sin(n \theta_{{x}}). 
\ee
By interchanging $x$ and $y$ and the fact that $G(x,y)=G(y,x)$, we have, for $|x|>|y|$,
\be\label{eq-green GPT expansion2}
G({x},y) = \sum_{n=0}^{\infty} \f{(-1)}{2\pi n}  \f{\cos(n \theta_x)}{ r_x^n} r_{{y}}^n\cos(n \theta_{{y}}) + \f{(-1)}{2\pi n} \f{\sin(n \theta_x)}{ r_x^n} r_{{y}}^n\sin(n \theta_{{y}}). 
\ee

If $x\in \p D_1$ and $y\in \p D_2$, then $|x|<|y|$. So, applying \eqref{eq-green GPT expansion1} gives
\beas
 \df{\p \mathcal{S}_{D_2}[\varphi_l]}{\p \nu_1}({x}) &=& \df{\p}{\p \nu_1}\int_{\p D_2} G({x},y) \varphi_ld\sigma(y)\\
&=&  \sum_{n=1}^{\infty} \f{\p r_{{x}}^n\cos(n \theta_{{x}}) }{\p \nu_1}a_{n,c}^l + \f{\p r_{{x}}^n\sin(n\theta_{{x}}) }{\p \nu_1}a_{n,s}^l .
\eeas
On the contrary, if $y\in \p D_1$ and $x\in \p D_2$,  then $|x|>|y|$. We have from \eqref{eq-green GPT expansion2} that, for any $f$,
\beas
\mathcal{S}_{D_1}[f](x) &=& \int_{\p D_1} G(x, {y})[f]( {y}) d\sigma( {y})\\
&=& \sum_{m=0}^{\infty}-\f{1}{2\pi m} \f{\cos(m \theta_x)}{ r_x^m} \int_{\p D_1} r_{ {y}}^m\cos(m \theta_{ {y}}) [f]( {y})d\sigma( {y}) \\
&& +\sum_{m=0}^{\infty}-\f{1}{2\pi m} \f{\sin(m \theta_x)}{ r_x^m} \int_{\p D_1} r_{ {y}}^m\sin(m \theta_{ {y}}) [f]( {y})d\sigma( {y}).
\eeas
Therefore, from the definition of $M_{m,n}$, we get
\beas
R_{jl} &=&   \left(\f{1}{2}-\lambda_j\right) \big(\mathcal{S}_{D_1}\left(\lambda_{D_1}Id - \mathcal{K}_{D_1}^*\right)^{-1}\df{\p \mathcal{S}_{D_2}[\varphi_l]}{\p \nu_1}, \varphi_j \big)_{\f{1}{2},-\f{1}{2}}\\
 &=& \left(\f{1}{2}-\lambda_j\right)\sum_{m= 0, n = 1}^{\infty} (a_{m,c}^j,a_{m,s}^j) M_{m,n}(\lambda_{D_1},D_1)(a_{n,c}^l,a_{n,s}^l)^t .
\eeas
For any $\lambda \in \mathbb{C}$ and $D=\delta B$, it is easy to check that $M_{m,n}(\lambda,D) = \delta^{m+n}M_{m,n}(\lambda,B)$. 
Since $D_2$ is in the intermediate regime, $a_{n,c}^l$ and $a_{n,s}^l$ satisfy
$$
|a_{m,c}^j|, |a_{m,s}^j| \leq \frac{1}{m}C^{-m}, \quad 
|a_{n,c}^l|, |a_{n,s}^l| \leq \frac{1}{n}C^{-n},
$$
for some constant $C>1$ independent of $\delta$.
Moreover, it can be shown that (see \cite{GPTs})
$$\sum_{n = 1}^{\infty} (a_{0,c}^j,a_{0,s}^j) M_{0,n}(\lambda_{D_1}, D_1)(a_{n,c}^l,a_{n,s}^l)^t = 0.$$
Then the conclusion immediately follows.
\end{proof}

\begin{cor}\label{cor_Pj_CGPT}
We have
\begin{align*}
\mathcal{P}_{j}(z) & -\sum_{l\neq j}\f{R_{jl}(z)R_{lj}(z)}{\lambda_j-\lambda_l} - 
\sum_{(l_1,l_2)\neq j}\f{R_{jl_2}R_{l_2l_1}R_{l_1j}}{(\lambda_j-\lambda_{l_1})(\lambda_j-\lambda_{l_2})}
\dots
\\
&  =
\left(\f{1}{2}-\lambda_j\right)\sum_{m= 1}^{M} \sum_{ n = 1}^{N} a_{m}^j M_{m,n}(\lambda_{D_1},D_1)(a_{n}^l)^t+ O(\delta^{M+N+1}).
\end{align*}
In the LHS, the summation should be truncated 
so that all the terms which contain
$R_{j l_k}\cdots R_{l_k j}=O(\delta^{2(k+1)})$
with $2(k+1)\leq M+N+1$ are ignored.
\end{cor}

\section{The inverse problem} \label{sec-inverse}

In this section, we consider the inverse problem associated with the forward system (\ref{eq-Helmholtz eq biosensing}). We assume that the plasmonic particle $D_2$ is known, {\it i.e.}, we know its electric permittivity $\eps_2=\eps_2(\omega)$, its shape $D_2$ and position $z$. The ordinary particle $D_1$ is unknown. For simplicity, we assume that its permittivity $\eps_1$ is known.  For each of many different positions $z$ of the plasmonic particle $D_2$, we measure the resonant frequency and use these resonant frequencies to reconstruct the shape of the ordinary particle $D_1$. 

As illustrated by Theorem \ref{thm-far field sensing}, the resonance in the scattered field occurs when $\lambda_{D_2}(\omega) - \lambda_j +  \mathcal{P}_j$ is minimized and $(\nu_l,\varphi_j)_{\mathcal{H}^*} (\varphi_j,x_m)_{-\f{1}{2},\f{1}{2}} \neq 0$. 
So by varying the frequency $\omega$, we can measure the value of $\lambda_j - \mathcal{P}_j$.
Moreover, in the absence of the ordinary particle, the resonance occurs when $\lambda_{D_2}(\omega) - \lambda_j$ is minimized and $(\nu_l,\varphi_j)_{\mathcal{H}^*} (\varphi_j,x_m)_{-\f{1}{2},\f{1}{2}} \neq 0$. Since we assume that the plasmonic particle $D_2$ is known, we can get the value of $\lambda_j$ a priori.
Therefore, by comparing $\lambda_j-\mathcal{P}_j$ and $\lambda_j$, we can measure the shift $\mathcal{P}_j$ of the eigenvalue. 

Finding $\mathcal{P}_j$ for many different positions of $D_2$ will yield a linear system of equations that will allow the recovery of the CGPTs associated with $D_1$.
From the recovered CGPTs, we will reconstruct the ordinary particle $D_1$. Here, we only consider the shape reconstruction problem. Nevertheless, by using the CGPTs associated with $D_1$, it is possible to reconstruct the permittivity $\eps_1$ of $D_1$ in the case it is not a priori given \cite{book3}. 

From now on, we denote $M_{m,n} = M_{m,n}(\lambda_{D_1},D_1)$.

\subsection{Contracted GPTs recovery algorithm}\label{algo-sensing}
We propose a recurrent algorithm to recover the GPTs of order less or equal to $k$ up to an order $\delta^{2k-1}$, using measurements of $P_j$ at different positions of $D_2$. For simplicity, we only consider the shift of a single eigenvalue $\lambda_j$ with a fixed $j$.
To gain robustness and efficiency, the shift in other resonant frequencies could also be considered.

We now explain our method for reconstructing GPTs $M_{m,n}, m+n\leq K$ for a given $K\in\mathbb{N}$ from the measurements of the shift $\mathcal{P}_j$.

Suppose we measure precisely $\mathcal{P}_j$  for three different positions $z_1,z_2,z_3$ of the plasmonic particle $D_2$.
First we reconstruct $M_{1,1}$ approximately. Since $M_{1,1}^t=M_{1,1}$, the matrix $M_{1,1}$ is symmetric. We look for a symmetric matrix $M_{1,1}^{(2)}$ satisfying
\beas
\mathcal{P}_{j}(z_1) &=&   \left(\f{1}{2}-\lambda_j\right) a_{1}^j(z_1) M_{1,1}^{(2)}(a_{1}^j)^t(z_1)\\
\mathcal{P}_{j}(z_2) &=&   \left(\f{1}{2}-\lambda_j\right) a_{1}^j(z_2) M_{1,1}^{(2)}(a_{1}^j)^t(z_2)\\
\mathcal{P}_{j}(z_{3}) &=&   \left(\f{1}{2}-\lambda_j\right) a_{1}^j(z_{3}) M_{1,1}^{(2)}(a_{1}^j)^t(z_{3}).
\eeas
The above equations can be seen as a linear system of equations for three independent components $(M_{1,1}^{(2)})_{11}, (M_{1,1}^{(2)})_{12}$ and $(M_{1,1}^{(2)})_{22}$.
We emphasize that $a_m^j(z_i)$ can be a priori given because the particle $D_2$ is known.
Since, from Corollary \ref{cor_Pj_CGPT} and the fact that $R_{jl}=O(\delta^2)$, we have
$$
\mathcal{P}_j(z_k) =  \left(\f{1}{2}-\lambda_j\right) a_{1}^j(z_k) M_{1,1}(a_{1}^j)^t(z_k) + O(\delta^3), \quad k=1,2,3,
$$
we see that $M_{1,1}$ is well approximated by $M_{1,1}^{(2)}$. Specifically, we have
$M_{1,1} - M_{1,1}^{(2)} = O(\delta^3)$.

Next we reconstruct and update the higher order GPTs $M_{n,m}$ in a recursive way. 
Towards this, we need more measurement data of the shift $\mathcal{P}_j$.
Let $k\geq 3$. 
Due to the symmetry of harmonic combinations of the non contracted GPTs (see \cite{book2}), we have $M_{m,n} = M_{n,m}^t$.
One can see that, by using this symmetry property, the set of GPTs $M_{m,n}$ satisfying ${m+n\leq k}$ contains $e_k$ independent variables where $e_k$ is given by
\beas
e_k = \left\{ \begin{array}{l} 
k(k-1) + k/2,\quad \mbox{ if } k \mbox{ is even},\\
k(k-1) + (k-1)/2,\quad \mbox{ if } k \mbox{ is odd}.
\end{array}\right.
\eeas
Therefore, we need $e_k$ measurement data for $\mathcal{P}_j$ to reconstruct the GPTs $M_{m,n}$ for $m+n\leq k$.

Suppose we have $e_k-2$ more measurement data $\mathcal{P}_j$ at different positions $z_4,z_5,...,z_{e_k}$. 
 Let $\{M_{m,n}^{(k)}\}_{m+n \leq k }$ be the set of matrices satisfying $[M^{(k)}_{n,m}]^t=M^{(k)}_{m,n}$ and the following linear system:
\begin{align}
\widetilde{\mathcal{P}}^{(k-1)}_{j}(z_1) &=   \left(\f{1}{2}-\lambda_j\right)\sum_{m + n \leq k} a_{m}^j(z_1) M_{m,n}^{(k)}(a_{n}^j)^t(z_1)
\notag
\\
\widetilde{\mathcal{P}}^{(k-1)}_{j}(z_2) &=   \left(\f{1}{2}-\lambda_j\right)\sum_{m + n \leq k} a_{m}^j(z_2) M_{m,n}^{(k)}(a_{n}^j)^t(z_2)
\notag
\\
\vdots \quad \quad &=\qquad\qquad\qquad\qquad\vdots 
\notag
\\
\widetilde{\mathcal{P}}^{(k-1)}_{j}(z_{e_k}) &=   \left(\f{1}{2}-\lambda_j\right)\sum_{m + n \leq k} a_{m}^j(z_{e_k}) M_{m,n}^{(k)}(a_{n}^j)^t(z_{e_k}),
\label{Pjk_linear_system}
\end{align}
where
\be\label{def_Pjk}
\widetilde{\mathcal{P}}^
{(k-1)}_{j}(z_i):=\mathcal{P}_{j}(z_i)-\sum_{l\neq j}\f{R_{jl}^{(k-1)}(z_i)R_{lj}^{(k-1)}(z_i)}{\lambda_j-\lambda_l} - \dots, \quad i=1,2,...,e_k,
\ee
and
\beas
R_{jl}^{(k-1)}(z) := \left(\f{1}{2}-\lambda_j\right)\sum_{ m + n \leq k-1} a_{m}^j(z) M_{m,n}^{(k-1)}(a_{n}^l)^t(z).
\eeas
Note that $M_{m,n}^{(k)}$ are defined recursively. In \eqref{def_Pjk}, the summation should be truncated as in Corollary \ref{cor_Pj_CGPT}.

Then $M_{m,n}^{(k)}$ becomes a good approximation of the GPT $M_{m,n}$ for $m+n\leq k$. Moreover, the accuracy improves as the iteration goes on. Indeed, we can see that
\be\label{claim_Mmnk}
M_{m,n} - M_{m,n}^{(k)} = O(\delta^{2k-1}),\quad m+n\leq k.
\ee
In fact, \eqref{claim_Mmnk} can be verified by induction. 
We already know that this is true when $k=2$. Let us assume 
$M_{m,n} - M_{m,n}^{(k-1)} = O(\delta^{2k-3})$, $m+n \leq k-1$.
Then, from Proposition \ref{prop-GPT perturbation}, we have
\beas
R_{jl}(z) - R_{jl}^{(k-1)}(z) &=& O(\delta^{2k-3}).
\eeas
Hence, from Corollary \ref{cor_Pj_CGPT} and the fact that $R_{jl}=O(\delta^2)$, we obtain
$$
\widetilde{\mathcal{P}}^
{(k-1)}_{j}(z_i)-\left(\mathcal{P}_{j}(z_i)-\sum_{l\neq j}\f{R_{jl}(z_i)R_{lj}(z_i)}{\lambda_j-\lambda_l} - \cdots\right) = O(\delta^{2k-1}).
$$
Therefore, in view of Corollary \ref{cor_Pj_CGPT} and the linear system \eqref{Pjk_linear_system}, we obtain \eqref{claim_Mmnk}. 
In conclusion, $M_{m,n}^{(k)}$ is indeed precise up to an order $\delta^{2k-1}$.

\begin{rmk} In practice, $\mathcal{P}_j$ might be subject to noise and could not be measured precisely. In this case only the low order CGPTs could be recovered.
\end{rmk}

\subsection{Shape recovery from contracted GPTs}
To recover the shape of $D_1$ from its contracted GPTs, we search to minimize the following shape functional (\cite{book3})
\be \label{min2c}
 \mathcal{J}_c^{(l)}[B]:= \frac{1}{2} \sum_{n+ m \le k}
 \left| N^{(1)}_{mn} (\lambda_{D_1}, B) -  N^{(1)}_{mn} (\lambda_{D_1}, D_1)
 \right|^2 \;,
\ee
where
\beas
N^{(1)}_{m,n} (\lambda, D) = (M_{m,n}^{cc} - M_{m,n}^{ss}) + i (M_{m,n}^{cs}-M_{m,n}^{sc} ).
\eeas
To minimize $\mathcal{J}^{(l)}[B]$ we need to compute the
\index{shape derivative} shape derivative, $d_S
\mathcal{J}_c^{(l)}$, of $\mathcal{J}_c^{(l)}$.

For $\epsilon$ small, let
$B_\epsilon$ be an $\epsilon$-deformation of $B$, {\it i.e.}, there is a 
scalar function $h \in \mathcal{C}^{1}(\p B)$, such that
 \beas
 \partial B_{\epsilon}:=\{x+\epsilon h(x)\nu(x)\ :  x \in \p B\} . 
 \eeas
Then, according to
\cite{book3,gpt1,gpt2}, the perturbation of a harmonic sum of GPTs due to
the shape deformation is given as follows:
\beas
N^{(1)}_{m,n} (\lambda_{D_1}, B_{\epsilon}) -  N^{(1)}_{m,n} (\lambda_{D_1}, D_1) \\
\nm 
\ds =  \epsilon (k_{\lambda_{D_1}}-1) \int_{ \p B} h(x)\left[\f{\p u}{\p \nu}\Big|_{-} \f{\p v}{\p \nu}\Big|_{-}
 +\frac{1}{k_{\lambda_{D_1}}} \f
 {\p u}{\p T}\Big|_{-}\f{\p v}{\p T}\Big|_{-}\right](x)\, d\sigma(x)+O(\epsilon^2),
\eeas
where \begin{equation} \label{kpbd} k_{\lambda_{D_1}} = (2\lambda_{D_1} +1)/
(2\lambda_{D_1} -1), \end{equation} and $u$ and $v$ are respectively
the solutions to the problems:
 \begin{equation}\label{u}
 \left\{
  \begin{array}{ll}
 \Delta u =0 \quad & \mbox{in } \ds B\cup (\R^2 \backslash \overline{B})\;,\\
 \nm
 \ds u|_{+} -u|_{-} =0 \quad &\mbox{on } \p B\;,\\
 \nm
 \ds \f{\p u}{\p \nu}\Big|_{+} -k_{\lambda_{D_1}} \f{\p u}{\p \nu}\Big|_{-} =0 \quad &\mbox{on } \p B\;,\\
 \nm
 \ds (u-(x_1 + ix_2)^m)(x)=O(|x|^{-1}) \quad &\mbox{as } |x|\rightarrow \infty\;,
 \end{array}
\right.
\end{equation}
and
\begin{equation}\label{v}
\left\{
  \begin{array}{ll}
 \Delta v =0 \quad & \mbox{in } \ds B\cup (\R^2 \backslash \overline{B})\;,\\
 \nm
 \ds k_{\lambda_{D_1}} v|_{+} -v|_{-} =0 \quad &\mbox{on } \p B\;,\\
 \nm
 \ds \f{\p v}{\p \nu}\Big|_{+} -\f{\p v}{\p \nu}\Big|_{-} =0 \quad &\mbox{on } \p B\;,\\
 \nm
 \ds (v-(x_1 + ix_2)^n)(x)=O(|x|^{-1}) \quad &\mbox{as } |x|\rightarrow \infty\;.
 \end{array}
\right.
\end{equation}
Here, $\partial /\partial T$ is the tangential derivative. 

 Let
\beas
w_{m,n}(x) = (k_{\lambda_{D_1}}-1) \left[\f{\p u}{\p \nu}\Big|_{-}
\f{\p v}{\p \nu}\Big|_{-}+\frac{1}{k_{\lambda_{D_1}}}
\f{\p u}{\p T}\Big|_{-}\f{\p v}{\p T}\Big|_{-}\right](x), \quad x \in \p
B\;.
\eeas
The shape derivative of $\mathcal{J}_c^{(l)}$ at $B$ in the direction of $h$
is given by
 \beas
 \langle d_S \mathcal{J}_c^{(l)}[B], h \rangle  = \sum_{m+n \leq k} \delta_{N}
 \langle w_{m,n}, h \rangle_{L^2(\partial B)} \;,
 \eeas
where
 $$
 \delta_{N} = N^{(1)}_{m,n} (\lambda_{D_1}, B) -  N^{(1)}_{m,n} (\lambda_{D_1}, D_1)\;.
 $$
Next, using a gradient descent algorithm we can minimize, at least locally, the functional $\mathcal{J}_c^{(l)}$. 

\section{Numerical Illustrations}  \label{sec-numeric}
In this section, we support our theoretical results by numerical examples. 
In the sequel, we assume that $D_2$ is an ellipse with semi-axes $a  = 1$ and $b = 2$, as shown in Figure \ref{fig-plasmonic ellipse}. In this case, as explained in Subsection \ref {subsec-CGPT}, the resonances in the far-field can only occur at $\lambda_1 = \f{1}{2}\f{a-b}{a+b} = -\f{1}{6}$ and $\lambda_2 = -\f{1}{2}\f{a-b}{a+b} = \f{1}{6}$. Thus, for a fixed position of $D_2$, we can measure two shifts of the plasmonic resonance: $\mathcal{P}_1$ and $\mathcal{P}_2$.
\begin{figure}[h]
\begin{center}
\includegraphics[scale=0.5]{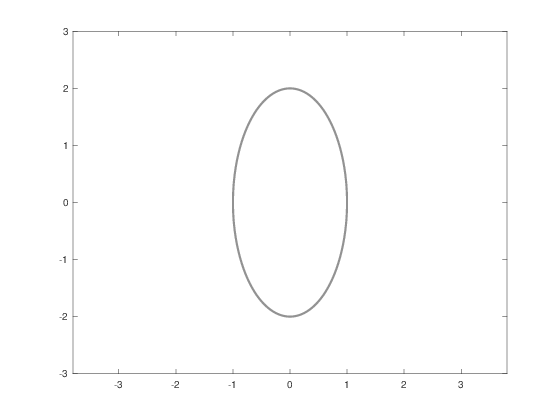}
\caption{ \label{fig-plasmonic ellipse} Plasmonic partcile $D_2$. }
\end{center}
\end{figure}

We consider the case of $D_1$ being a triangular-shaped and a rectangular-shaped particle with known contrast $\lambda_{D_1} = 1$, as shown in Figure \ref{fig-non plasmonic}.
\begin{figure}[h]
\begin{center}
\includegraphics[scale=0.4]{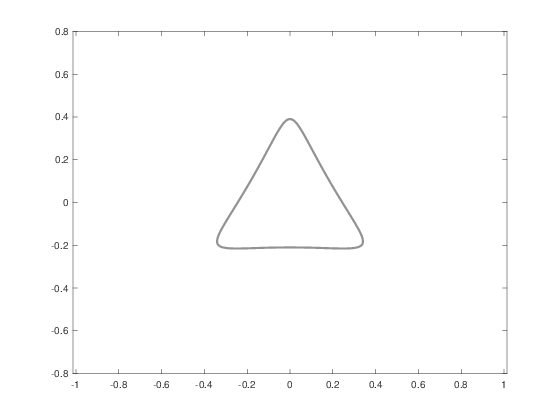}
\includegraphics[scale=0.4]{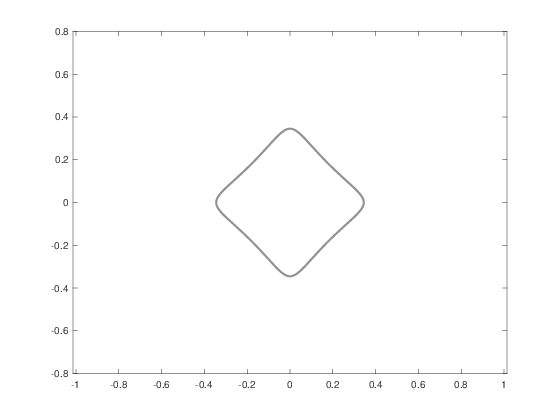}
\caption{ \label{fig-non plasmonic} Non plasmonic partciles $D_1$. Triangular-shaped (left) and rectangular-shaped (right).}
\end{center}
\end{figure}

Figure \ref{fig-shift in resonance} shows the shift in the plasmonic resonance around $\lambda_1$, for random positions of $D_2$ around a triangular-shaped particle $D_1$. From these measurements, $\mathcal{P}_1$ can be precisely estimated from the resonance peaks and the equation $\mathcal{P}_j = \lambda_j - \lambda_r$, where $\lambda_r$ is the value at which we achieve the maximum of the resonant peak. 

It is worth  mentioning that, for the sake of simplicity and clarity, we plot the graph not by varying the frequency but the parameter $\lambda$ directly. We assume $\mbox{Re}(\lambda_{D_2})$ ranges from $-1/2$ to $1/2$ and $\mbox{Im}(\lambda_{D_2}) =  10^{-4}$. In a more realistic setting, corrections in the peaks of resonances should be included, by considering the Drude model for $\lambda_{D_2}$. But they are essentially equivalent.
\begin{figure}[h]
\begin{center}
\includegraphics[scale=0.7]{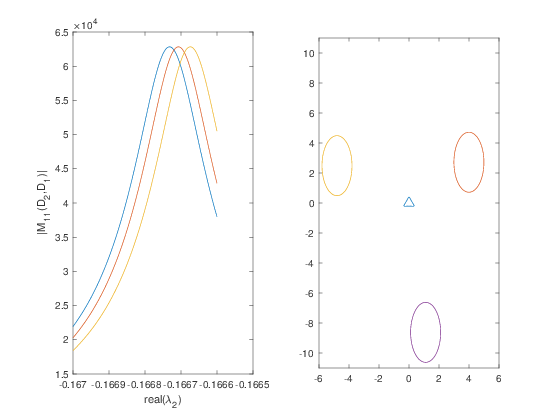}
\caption{ \label{fig-shift in resonance} (right) Modulus of the entry (1,1) of the first order polarization tensor given in Theorem \ref{thm-far field sensing}, for different positions of $D_2$ around a triangular-shaped particle $D_2$ (left). }
\end{center}
\end{figure}

To recover geometrical properties of $D_1$ from measurements of $\mathcal{P}_1$, we recover the contracted GPTs using the algorithm described in $\ref{algo-sensing}$ and then minimize functional \eqref{min2c} to reconstruct an approximation of $D_1$.

To recover the first contracted GPTs of order 5 or less we make 22 measurements around $D_1$ as shown in Figure \ref{fig-22 measurements},  and measure the shift from $\lambda_1 = -\f{1}{6}$.
\begin{figure}[h]
\begin{center}
\includegraphics[scale=0.4]{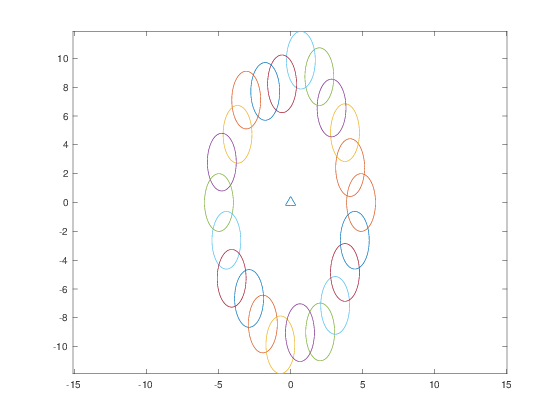}
\includegraphics[scale=0.4]{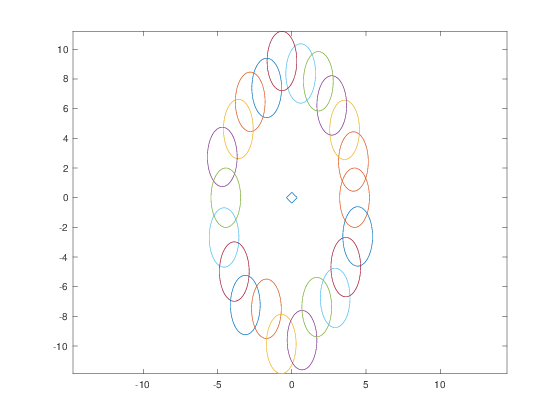}
\caption{ \label{fig-22 measurements} Positions of $D_2$ for which we measure $\mathcal{P}_1$. (left) Triangular-shaped particle $D_1$, (right) rectangular-shaped particle $D_1$.}
\end{center}
\end{figure}

In the following, we show a comparison between the recovered contracted GPTs of order less or equal to 4 and their theoretical values, for each iteration.
\\ 

\underline{Triangle-shaped $D_1$}:
\\ \\
Theoretical values:
\beas
&& M_{11}= \left(\begin{array}{c c}
0.2426 &  0\\
   0  &  0.2426
\end{array}\right) \quad M_{12}= \left(\begin{array}{c c}
0  &  -0.0215\\
   -0.0215  &  0
\end{array}\right)\quad M_{22}= \left(\begin{array}{c c}
0.043 &  0\\
   0  &  0.043
\end{array}\right)\\
&& M_{13}= \left(\begin{array}{c c}
0 &  0\\
   0  &  0
\end{array}\right);
\eeas
Recovered:
\beas
&& M_{11}^{(2)} = \left(\begin{array}{c c}
0.2444 &  -0.0007\\
   -0.0007  &  0.2408
\end{array}\right)\quad    M_{11}^{(3)} = \left(\begin{array}{c c}
0.2438  & 0  \\
   0 &  0.2414
\end{array}\right) \quad  M_{11}^{(4)} = \left(\begin{array}{c c}
0.2429 &  -0.0001\\
 -0.0001 & 0.2430
\end{array}\right) \\
&& M_{11}^{(5)} = \left(\begin{array}{c c}
0.2426 &  0\\
   0  &  0.2426
\end{array}\right)
\eeas
\beas
M_{12}^{(3)} = \left(\begin{array}{c c}
   0.0008  &  -0.2414  \\
    -0.0212  &  -0.0087
\end{array}\right)\quad    M_{12}^{(4)} = \left(\begin{array}{c c}
   0 &  -0.2413  \\
    -0.0213  &  0
\end{array}\right) \quad  M_{12}^{(5)} = \left(\begin{array}{c c}
   0 &  -0.2415  \\
    -0.0215  &  0
\end{array}\right) 
\eeas
\beas
&& M_{22}^{(4)} = \left(\begin{array}{c c}
0.0180  &  0.2204\\
   0.2204  &  0.0389 
\end{array}\right)\quad    M_{22}^{(5)} = \left(\begin{array}{c c}
0.0368  &  0.0010\\
   0.0010  &  0.0497
\end{array}\right) \quad  M_{13}^{(4)} = \left(\begin{array}{c c}
0.0093 &  -0.1126\\
   -0.1123 &  -0.0019
\end{array}\right)\\
 && M_{13}^{(5)} = \left(\begin{array}{c c}
0.0032 &  -0.0005\\
-0.0005  & -0.0032
\end{array}\right).
\eeas

\underline{Rectangular-shaped $D_1$}:
\\ \\
Theoretical values:
\beas
&& M_{11}= \left(\begin{array}{c c}
0.2682  &  0.0000\\
   0  &  0.2682
\end{array}\right) \quad M_{12}= \left(\begin{array}{c c}
0  &  0\\
   0  &  0
\end{array}\right)\quad M_{22}= \left(\begin{array}{c c}
0.0544 &  0\\
   0  &  0.0402
\end{array}\right)\\
&& M_{13}= \left(\begin{array}{c c}
0.0054 &  0\\
   0  &  -0.0054
\end{array}\right);
\eeas
Recovered:
\beas
&& M_{11}^{(2)} = \left(\begin{array}{c c}
    0.2703  &  0.0001\\
    0.0001  &  0.2661
\end{array}\right)\quad    M_{11}^{(3)} = \left(\begin{array}{c c}
0.2696  &  0  \\
0  &  0.2662
\end{array}\right) \quad  M_{11}^{(4)} = \left(\begin{array}{c c}
0.2682 & 0\\
0 & 0.2681
\end{array}\right) \\
&& M_{11}^{(5)} = \left(\begin{array}{c c}
0.2682  &  0\\
0  &  0.2681
\end{array}\right)
\eeas
\beas
M_{12}^{(3)} = \left(\begin{array}{c c}
   0.0038 &  -0.0001  \\
   0  & -0.0112
\end{array}\right)\quad    M_{12}^{(4)} = \left(\begin{array}{c c}
   0 &  0  \\
    0  &  0
\end{array}\right) \quad  M_{12}^{(5)} = \left(\begin{array}{c c}
   0 &  0  \\
    0  &  0
\end{array}\right) 
\eeas
\beas
&& M_{22}^{(4)} = \left(\begin{array}{c c}
0.0530  &  -0.0007\\
-0.0007  &  0.0425 
\end{array}\right)\quad    M_{22}^{(5)} = \left(\begin{array}{c c}
0.0537  &  0.0006\\
0.0006  &  0.0416
\end{array}\right) \quad  M_{13}^{(4)} = \left(\begin{array}{c c}
0.0064  &  0.0003\\
0.0004 &  -0.0063
\end{array}\right)\\
 && M_{13}^{(5)} = \left(\begin{array}{c c}
0.0060  & -0.0003\\
-0.0003  & -0.0059
\end{array}\right).
\eeas

The results of minimizing the functional \eqref{min2c} with a gradient descent approach and using the recovered contracted GPTs of order less or equal to 5 are shown in Figures \ref{fig-shape recovery triangle} and \ref{fig-shape recovery rectanle}. We take as initial point the equivalent ellipse to $D_1$, given by the first order polarization recovered with Algorithm \ref{algo-sensing}, i.e $M_{11}^{(5)}$.
\begin{figure}[h]
\begin{center}
\includegraphics[scale=0.25]{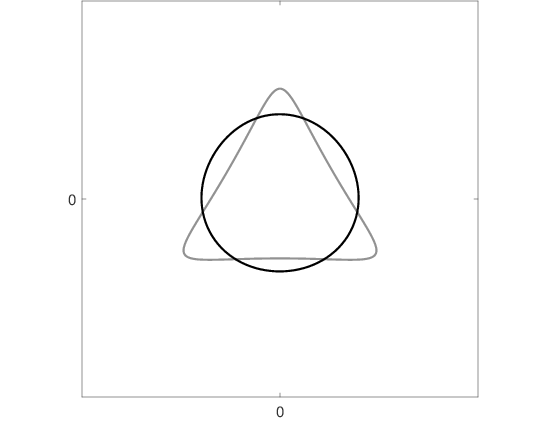}
\includegraphics[scale=0.25]{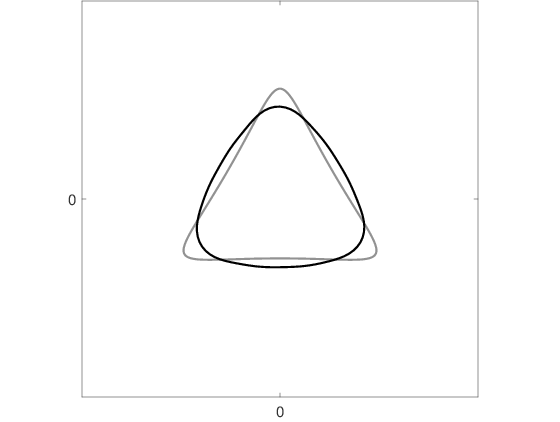}
\includegraphics[scale=0.25]{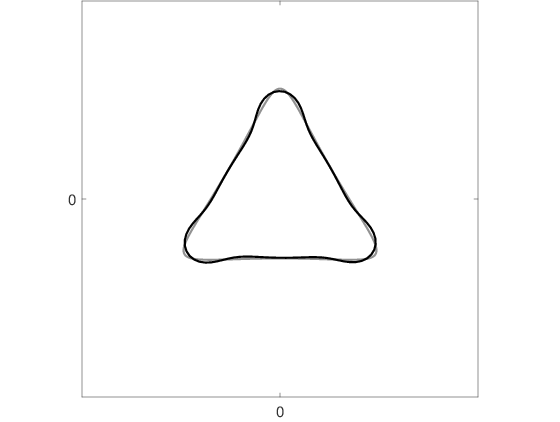}
\caption{ \label{fig-shape recovery triangle} Shape recovery of a triangular-shaped particle $D_1$. From left to right, we show both, the original shape and the recovered one after 0 iterations, after 8 iterations and after 30 iterations.}
\end{center}
\end{figure}

\begin{figure}[h]
\begin{center}
\includegraphics[scale=0.25]{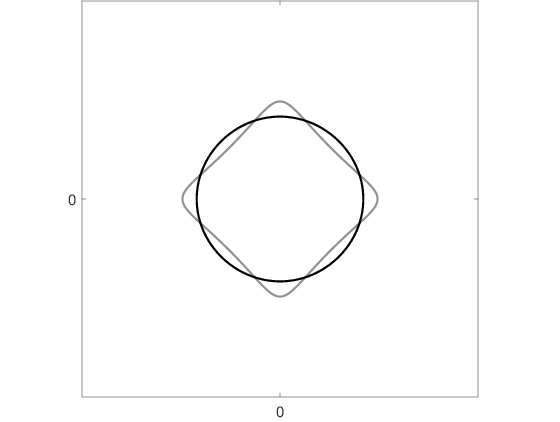}
\includegraphics[scale=0.25]{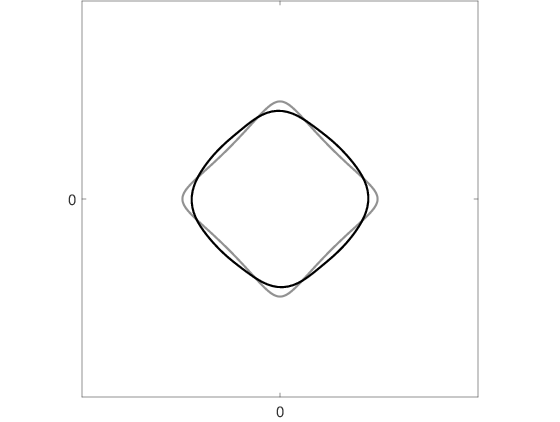}
\includegraphics[scale=0.25]{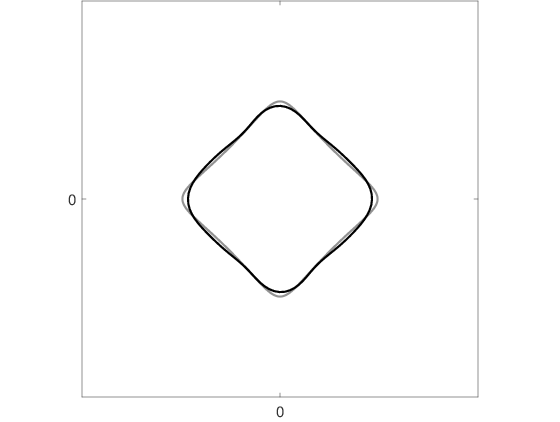}
\caption{ \label{fig-shape recovery rectanle} Shape recovery of a rectangular-shaped particle $D_1$. From left to right, we show both, the original shape and the recovered one after 0 iterations, after 30 iterations and after 100 iterations.}
\end{center}
\end{figure}
\section{Conclusion} \label{sec-conclusion}
In this paper, using the quasi-static model, we have shown that the fine details of a small object can be reconstructed from the shift of resonant frequencies it induces to a plasmonic particle in the intermediate regime. This provides a solution for the ill-posed inverse problem of reconstructing small objects from far-field measurements and also laid a mathematical foundation for plasmonic bio-sensing.  The idea can be extended in several directions: (i) to investigate the strong interaction regime when the small object is close to the plasmonic particle; (ii) to study the case when the size of object is comparable to the size of plasmonic particle; (iii) to analyze the case with multiple small objects and multiple plasmonic particles; (iv) to consider the more practical model of Maxwell equations, and (v) to investigate other types of subwavelength resonances such as Minnaert resonance \cite{Ammari2016_Minnaert, Minnaert1933} in bubbly fluids. These new developments will be reported in forthcoming works.

\end{document}